\documentclass[12pt]{article}      
\usepackage{amsmath,amssymb,amsthm, amscd, graphicx, verbatim, setspace} 
\usepackage{setspace}

\theoremstyle{plain}
\newtheorem{theorem}{Theorem}
\newtheorem{lemma}[theorem]{Lemma}
\newtheorem{corollary}[theorem]{Corollary}

\theoremstyle{definition}

\newtheorem{question}[theorem]{Question}

\theoremstyle{remark}

\title{Improved lower bounds on extremal functions of multidimensional permutation matrices}
\date{}
\author{Jesse Geneson\\
\small Department of Mathematics\\[-0.8ex]
\small MIT\\[-0.8ex]
\small Massachusetts, U.S.A.\\
\small\tt geneson@math.mit.edu\\
}
\begin{document}
\maketitle

\begin{abstract}
A $d$-dimensional zero-one matrix $A$ avoids another $d$-dimensional zero-one matrix $P$ if no submatrix of $A$ can be transformed to $P$ by changing some ones to zeroes. Let $f(n,P,d)$ denote the maximum number of ones in a $d$-dimensional $n \times  \cdots \times n$ zero-one matrix that avoids $P$. 

Fox proved for $n$ sufficiently large that $f(n, P, 2) = 2^{k^{\Theta(1)}}n$ for almost all $k \times k$ permutation matrices $P$. We extend this result by proving for $d \geq 2$ and $n$ sufficiently large that $f(n, P, d) = 2^{k^{\Theta(1)}}n^{d-1}$ for almost all $d$-dimensional permutation matrices $P$ of dimensions $k \times \cdots \times k$.
\end{abstract}

\section{Introduction}\label{5c}
 
An early motivation for bounding matrix extremal functions was to use them for solving problems in computational and discrete geometry \cite{BG, Fure, Mit}. Mitchell wrote an algorithm to find a shortest rectilinear path that avoids obstacles in the plane \cite{Mit} and proved that the complexity of this algorithm is bounded from above in terms of a specific matrix extremal function, which was bounded by Bienstock and Gy\"ori \cite{BG}.  F\"{u}redi \cite{Fure} used matrix extremal functions to derive an upper bound on Erd\H{o}s and Moser's \cite{EM} problem of maximizing the number of unit distances in a convex $n$-gon. Recent interest in the extremal theory of matrices has been spurred by the resolution of the Stanley-Wilf conjecture using the linearity of the extremal functions of forbidden permutation matrices \cite{Kl, MT}.

A $d$-dimensional $n_1 \times \cdots \times n_d$ matrix is denoted by $A = \left(a_{i_1, \ldots , i_d}\right)$, where $1 \le i_l \le n_{\ell}$ for $\ell = 1, 2, \ldots , d$. An \emph{$\ell$-cross section} of matrix $A$ is a maximal set of entries $a_{i_1,\ldots , i_d}$ with $i_{l}$ fixed. An \emph{$\ell$-row} of matrix $A$ is a maximal set of entries $a_{i_1,\ldots , i_d}$ with $i_{j}$ fixed for every $j \neq l$. 

The Kronecker product of two $d$-dimensional $0-1$ matrices $M$ and $N$, denoted by $M \otimes N$, is the $d$-dimensional matrix obtained by replacing each $1$ in $M$ with a copy of $N$ and each $0$ in $M$ with a $0$-matrix the same size as $N$. A $d$-dimensional $k  \times \cdots \times k$ $0-1$ matrix is a \emph{permutation matrix} if each of its $l$-cross sections contains a single one for every $\ell=1, \ldots , d$.

We say that a $d$-dimensional $0-1$ matrix $A$ \emph{contains} another $d$-dimensional $0-1$ matrix $P$ if $A$ has a submatrix that can be transformed into $P$ by changing some number of ones to zeroes. Otherwise, $A$ is said to \emph{avoid} $P$.  Let $f(n, P, d)$ be the maximum number of ones in a $d$-dimensional $n \times \cdots \times n$ $0-1$ matrix that avoids a given $d$-dimensional $0-1$ matrix $P$. It is easy to obtain trivial lower and upper bounds of $n^{d-1} \leq f(n,P,d) \leq n^{d}$ if $P$ has at least two one entries.

The upper bound in the last inequality is a factor of $n$ higher than the lower bound. Most work on improving this bound has been for the case $d = 2$. Pach and Tardos proved that $f(n, P, 2)$ is super-additive in $n$ \cite{PT}. By Fekete's Lemma on super-additive sequences \cite{Fekete}, the sequence $\{ {f(n, P, 2) \over n} \}$ is convergent. The limit is known as the F\"uredi-Hajnal limit \cite{C, Fox}. Cibulka \cite{C} showed that this limit is at least $2(k-1)$ when $P$ is a $k \times k$ permutation matrix and that the limit is exactly $2(k-1)$ when $P$ is the $k \times k$ identity matrix. 

Marcus and Tardos \cite{MT} showed that the F\"uredi-Hajnal limit has an upper bound of $2^{O(k \log k)}$ for every $k \times k$ permutation matrix $P$, and Fox \cite{Fox} improved this upper bound to $2^{O(k)}$. Klazar and Marcus \cite{KM} bounded the extremal function when the $d$-dimensional matrix $P$ is a permutation matrix of size $k \times \cdots \times k$, generalizing the $d = 2$ result \cite{MT} by proving that $f(n,P,d)=O(n^{d-1})$. In particular, they showed that ${f(n, P, d) \over n^{d-1}} = 2^{O(k \log k)}$, which generalizes Marcus and Tardos' upper bound on the F\"uredi-Hajnal limit \cite{ MT}. 

For each fixed $d \geq 2$ there is an upper bound of $2^{O(k)}$ on ${f(n,P,d) \over n^{d-1}}$ for every $d$-dimensional permutation matrix $P$ of dimensions $k \times \cdots \times k$ \cite{Fox, GT}. We show for $d \geq 2$ and $n$ sufficiently large that ${f(n,P,d) \over n^{d-1}}$ has a lower bound of $2^{\Omega(k^{1 / d})}$ for a family of $k \times \cdots \times k$ permutation matrices. A similar result was proved for a different family of permutation matrices in \cite{GT}. 

As a corollary we show for $d \geq 2$ and $n$ sufficiently large that $f(n, P, d) = 2^{k^{\Theta(1)}}n^{d-1}$ for almost all $d$-dimensional permutation matrices $P$ of dimensions $k \times \cdots \times k$. Bounding the extremal function of almost all $d$-dimensional permutation matrices of dimensions $k \times \cdots \times k$ was an open problem at the end of \cite{GT}.

\section{Lower bounds for $d$-dimensional permutation matrices}
\label{constant}

The proof uses the notions of cross section contraction and interval minor containment \cite{Fox}. Contracting several consecutive $l$-cross sections of a $d$-dimensional matrix means replacing these $l$-cross sections by a single $l$-cross section and placing a one in an entry of the new cross section if and only if at least one of the corresponding entries in the original $l$-cross sections is a one. We say that $A$ contains $B$ as an interval minor if we can use repeated cross section contraction to transform $A$ into a matrix which contains $B$. $A$ avoids $B$ as an interval minor if $A$ does not contain $B$ as an interval minor. 

The containment in previous sections is at least as strong as interval minor containment. Indeed, $A$ contains $B$ implies that $A$ contains $B$ as an interval minor. However, since a permutation matrix has only one $1$-entry in every cross section, containment of a permutation matrix $P$ is equivalent to containment of $P$ as an interval minor.

Analogous to $f(n, P, d)$, we define $m(n, P, d)$ to be the maximum number of ones in a $d$-dimensional $n \times \cdots \times n$ zero-one matrix that avoids $P$ as an interval minor. Let $R^{k_1, \ldots , k_d}$ be the $d$-dimensional matrix of dimensions $k_1 \times \ldots \times k_d$ with every entry equal to $1$. 

We observe that
\begin{equation}
\label{fm}
f(n, P, d) \leq m(n, R^{k, \ldots, k}, d) 
\end{equation}
for every $k \times \cdots \times k$ permutation matrix $P$. This follows from the fact that containment of $R^{k, \ldots, k}$ as an interval minor implies containment of $P$.

Define a corner entry of a $k_1 \times \cdots \times k_d$ matrix $P = (p_{i_1, \ldots, i_d})$ to be an entry $p_{i_1, \ldots, i_d}$ located at a corner of $P$, i.e., $i_{\tau}=1$ or $i_{\tau} = k_{\tau}$ for every $1 \leq \tau \leq d$.

\begin{theorem}
\label{lower}
Let $d \geq 2$ and suppose that $k^{1/d}$ is a multiple of $20$. For every $d$-dimensional $k\times \cdots \times k$ permutation matrix $P$ that has a corner $1$-entry and contains $R^{k^{1/d}, \ldots, k^{1/d}}$ as an interval minor, ${f(n,P,d) \over n^{d-1}}=2^{\Omega(k^{1 / d})}$ for $n \geq 2^{ k^{1/d}/20}$.
\end{theorem}

The theorem is proved in the next few lemmas. The first lemma was proved in \cite{GT}.

\begin{lemma}
\label{homo}\cite{GT}
If $d \geq 2$ and $P$ is a $d$-dimensional zero-one matrix with a corner $1$-entry, then $f(sn,P,d) \ge {s^{d-1} \over (d-1)!}f(n,P,d)$.
\end{lemma}

The next lemma is slightly different from a lemma in \cite{GT}.

\begin{lemma}
\label{converge}
If $d \geq 2$ and $P$ is a $d$-dimensional zero-one matrix with a corner $1$-entry, then for any positive integers $m$ and $n \geq m$,
$${f(n,P,d) \over n^{d-1}} \ge {1 \over 2^{d-1}(d-1)!} \ {f(m,P,d) \over m^{d-1}}.$$
\end{lemma}

\begin{proof}
Write $n$ as $n=sm+r$, with $0 \le r<m$. Then ${f(n,P,d) \over n^{d-1}} = {f(sm+r,P,d) \over (sm+r)^{d-1}} 
 \ge  {f(sm,P,d) \over (sm+r)^{d-1}} 
\ge {s^{d-1} \over (d-1)!}{f(m,P,d) \over (sm+r)^{d-1}} \ge {s^{d-1} \over (d-1)!}{f(m,P,d) \over ((s+1)m)^{d-1}}  \ge {1 \over 2^{d-1}(d-1)!} \ {f(m,P,d) \over m^{d-1}}$, where Lemma \ref{homo} is used in the second inequality.
\end{proof}

The following lemma was proved in \cite{GT} and was based on Fox's methods for the $d=2$ case in \cite{Fox2}.

\begin{lemma}
\label{exists}
\cite{GT}
For each $d \geq 2$ and $\ell$ that is a multiple of $20$, there exists an $N \times \cdots \times N$ zero-one matrix $A$, with $N=2^{\ell/20}$, that has $\Theta(N^{d-1/2})$ ones and avoids $R^{\ell,\ldots, \ell}$ as an interval minor.
\end{lemma}

Now we prove the first theorem.

\begin{proof} [Proof of Theorem \ref{lower}]

Let $\ell= k^{1/d}$ be a multiple of $20$, and let $P$ be any $d$-dimensional permutation matrix of size $k \times \cdots \times k$ that contains $R^{\ell, \ldots , \ell}$ as an interval minor and has at least one corner $1$-entry. Since $P$ contains $R^{\ell,\ldots,\ell}$ as an interval minor, $f(N,P,d)\ge m(N,R^{\ell, \ldots, \ell},d)$ for $N = 2^{\ell/20}$, which by Lemma \ref{converge} and Lemma \ref{exists} implies for $n \geq N$ that $${f(n,P,d) \over n^{d-1}}  \ge {1 \over 2^{d-1} (d-1)!} \ {m(N, R^{\ell, \ldots, \ell}, d) \over N^{d - 1}} 
= 2^{\Omega(k^{1 / d})}.
$$
\end{proof}

We extend the lower bound to a different class of matrices in the next corollary.

\begin{corollary}
Let $d \geq 2$ and $\ell=k^{1/d}$ such that $\ell-1$ is a multiple of $20$, and let $P$ be any $d$-dimensional permutation matrix of size $k \times \cdots \times k$ that contains $R^{\ell, \ldots , \ell}$ as an interval minor. If $N = 2^{(\ell-1)/20}$, then ${f(n,P,d) \over n^{d-1}}  = 2^{\Omega(k^{1 / d})}$ for $n \geq N$.
\end{corollary}

\begin{proof}
Since $P$ contains $R^{\ell, \ldots , \ell}$ as an interval minor, each one in $P$ corresponds to an entry of $R^{\ell, \ldots , \ell}$. Let $t$ be the corner entry of $R^{\ell, \ldots , \ell}$ with coordinates $(1, \cdots, 1)$. 

Let $P'$ be obtained from $P$ by deleting every one in $P$ that corresponds to an entry besides $t$ in $R^{\ell, \ldots , \ell}$ in the same $q$-cross section as $t$ for some $q = 1, \cdots, d$, deleting the one in $P$ that corresponds to the entry in $R^{\ell, \ldots , \ell}$ with coordinates $(2, \cdots, 2)$, and removing any cross sections with no ones. Then $P'$ contains $R^{\ell-1, \ldots, \ell-1}$ as an interval minor and $P'$ has a corner $1$-entry. Then $f(n,P,d) \ge f(n, P', d) = 2^{\Omega(k^{1 / d})} n^{d-1}$ for $n \geq  2^{(\ell-1)/20}$.
\end{proof}

\section{Sharp bounds on $m(n,R^{k, \ldots, k},d)$}

We proved that ${m(n,R^{k, \ldots, k},d) \over n^{d-1}} = 2^{O(k)}$ in \cite{GT}. In this section, we prove matching lower bounds on ${ m(n,R^{k, \ldots, k},d) \over n^{d - 1}}$ for $n \geq 2^{(k-1)/20}$ such that $k-1$ is a multiple of $20$. 

\begin{lemma}
\label{homo1}
$m(sn,R^{k, \ldots, k},d) \ge {s^{d-1} \over (d-1)!}m(n,R^{k-1, \ldots, k-1},d)$ for all $d \geq 2$, $s \geq 1$, and $n \geq 1$.
\end{lemma}
\begin{proof}
Let $M$ be an $s \times \cdots\times s$ matrix with ones at the coordinates $(i_1, \ldots, i_d)$ where $i_1+ \cdots+i_d=s+d-1$ and zeroes everywhere else, so that $M$ has ${s+d-2 \choose d-1} \ge {s^{d-1} \over (d-1)!}$ ones. Let $N$ be an $n \times \cdots \times n$ matrix that avoids $R^{k-1, \ldots, k-1}$ as an interval minor and has $m(n,R^{k-1, \ldots, k-1},d)$ ones. It suffices to prove that $M \otimes N$ avoids $R^{k, \ldots, k}$ as an interval minor.

Suppose for contradiction that $M \otimes N$ contains $R^{k, \ldots, k}$ as an interval minor. Let $r_{1, \ldots, 1}$ be the corner $1$-entry of $R^{k, \ldots, k}$ with every coordinate minimal, and pick an arbitrary $1$-entry $r^*$ in $R^{k, \ldots, k}$ other than $r_{1, \ldots, 1}$ such that $r^*$ and $r_{1, \ldots, 1}$ have no coordinates in common. Suppose that $r_{1, \ldots, 1}$ and $r^*$ are represented by $e_1$ and $e_2$  in $M \otimes N$, respectively. In particular, suppose that $e_1$ and $e_2$ are in the $S$-submatrices $S(i_1, \ldots, i_d)$ and $S(j_1, \ldots, j_d)$, respectively. Note that $i_1+ \cdots +i_d = j_1+ \cdots + j_d$. Since each coordinate of $r^*$ is greater than each coordinate of $r_{1, \ldots, 1}$ in an interval minor copy of $R^{k, \ldots, k}$, then each coordinate of $e_2$ must also be greater than each coordinate of $e_1$ in $M \otimes N$, and hence $i_{\tau} \leq j_{\tau}$ for $\tau = 1, 2, \ldots, d$. It then follows from $i_1+ \cdots +i_d = j_1+ \cdots + j_d$ that $i_{\tau} = j_{\tau}$ for $\tau= 1, 2, \ldots, d$, i.e., the two entries $e_1$ and $e_2$ must be in the same $S$-submatrix in $M \otimes N$. 

Since $r^*$ can be any arbitrary $1$-entry in $R^{k, \ldots, k}$ with no coordinates in common with $r_{1, \ldots, 1}$, a single $S$-submatrix contains $R^{k-1, \ldots, k-1}$ as an interval minor. However, this is a contradiction since each nonzero $S$-submatrix in $M \otimes N$ is a copy of $N$, which avoids $R^{k-1, \ldots, k-1}$ as an interval minor. Thus $M \otimes N$ avoids $R^{k, \ldots, k}$ as an interval minor.
\end{proof}

\begin{lemma}
\label{converge1}
${m(n,R^{k, \ldots, k},d) \over n^{d - 1}} \ge {1 \over 2^{d-1} (d-1)!} \ {m(t,R^{k-1, \ldots, k-1},d) \over t^{d-1}}$ for $d \geq 2$, $t \geq 1$, and $n \geq t$.
\end{lemma}

\begin{proof}
Suppose that $n=st+r$, with $0 \le r<t$. Then 
${m(n,R^{k, \ldots, k},d) \over n^{d-1}} \ge  {m(st,R^{k, \ldots, k},d) \over (st+r)^{d-1}} \ge {s^{d-1} \over (d-1)!}{m(t,R^{k-1, \ldots, k-1},d) \over (st+r)^{d-1}} \ge {1 \over 2^{d-1} (d-1)!} \ {m(t,R^{k-1, \ldots, k-1},d) \over t^{d-1}}$ by Lemma \ref{homo1}.
\end{proof}

\begin{corollary}
If $d \geq 2$ and $n \geq 2^{(k-1)/20}$ such that $k-1$ is a multiple of $20$, then ${ m(n,R^{k, \ldots, k},d) \over n^{d - 1}} = 2^{\Omega(k)}$.
\end{corollary}

\begin{proof}
Let $N = 2^{(k-1)/20}$. Then ${ m(n,R^{k, \ldots, k},d) \over n^{d - 1}} \ge {1 \over 2^{d-1} (d-1)!} \ {m(N,R^{k-1, \ldots, k-1},d) \over N^{d-1}} = 2^{\Omega(k)}$ for $n \geq N$ by Lemma \ref{exists} and Lemma \ref{converge1}.
\end{proof}

\section{Bounds for almost all $d$-dimensional permutation matrices}

Next we prove for $d \geq 2$ that $f(n, P, d) = 2^{k^{\Theta(1)}}n^{d-1}$ for almost all $d$-dimensional permutation matrices $P$ of dimensions $k \times \cdots \times k$ and for $n$ sufficiently large. The proof of the next lemma is much like the proof of the corresponding lemma for $d = 2$ in \cite{Fox}.

\begin{lemma}
For $d \geq 2$ and $k \geq (d+1) (2\ell)^d \ln{\ell}$, the probability that a random $d$-dimensional permutation matrix of dimensions $k \times \cdots \times k$ avoids $R^{\ell, \ldots , \ell}$ as an interval minor is at most $1/\ell$. 
\end{lemma}

\begin{proof}
Suppose that $k \geq (d+1) (2\ell)^d \ln{\ell}$. Then in a random $d$-dimensional permutation matrix $P$ of dimensions $k \times \cdots \times k$, the probability that a given $\lfloor k/\ell \rfloor \times \cdots \times \lfloor k/\ell \rfloor$ submatrix has all zeros is at most $(1-(1/\ell-1/k)^{d-1})^{k/\ell-1} < (1-\frac{1}{(2\ell)^{d-1}})^{\frac{k}{2\ell}} < e^{\frac{-k}{(2\ell)^{d}}} < \ell^{-(d+1)}$. Let $P'$ be a submatrix of $P$ of dimensions $\ell \lfloor k/\ell \rfloor \times \cdots \times \ell \lfloor k/\ell \rfloor$. If each dimension of $P'$ is partitioned into $\ell$ equal intervals of length $\lfloor k/\ell \rfloor$, then there are $\ell^{d}$ blocks, so the probability that $P$ avoids $R^{\ell, \ldots , \ell}$ as an interval minor is at most $\ell^{d}\ell^{-(d+1)} = 1/\ell$.
\end{proof}

\begin{corollary}
If $d \geq 2$ and $\ell = 20 \lfloor \frac{\frac{1}{2}(\frac{k}{(d+1)\ln{k}})^{1/d}-1}{20} \rfloor+1$, then $f(n, P, d) = 2^{k^{\Theta(1)}}n^{d-1}$ for $n \geq 2^{(\ell-1)/20}$ for almost all $d$-dimensional permutation matrices $P$ of dimensions $k \times \cdots \times k$.
\end{corollary}

\begin{proof}
The upper bound was proved for all $d$-dimensional permutation matrices in \cite{GT}. For the lower bound observe that $\ell-1$ is a multiple of $20$ and $k \geq (d+1) (2\ell)^d \ln{\ell}$, so the probability that a random $d$-dimensional permutation matrix of dimensions $k \times \cdots \times k$ avoids $R^{\ell, \ldots , \ell}$ as an interval minor is at most $1/\ell$. Then for almost all $d$-dimensional permutation matrices $P$ of dimensions $k \times \cdots \times k$, $f(n, P, d) \geq m(n,R^{\ell, \ldots, \ell},d) = 2^{k^{\Omega(1)}}n^{d-1}$ for $n \geq 2^{(\ell-1)/20}$.
\end{proof}

\section{Open Problems}\label{narf}

Bounds on $f(n, P, d)$ are unknown for most $d$-dimensional $0-1$ matrices $P$. Besides finding bounds on $f(n, P, d)$ for specific $d$-dimensional $0-1$ matrices $P$, there are a few more general open questions about $f(n, P, d)$ and $m(n, P, d)$.

\begin{question}
Does the sequence $\{ {f(n, P, d) \over n^{d -1 }} \}$ converge for every $d$-dimensional $0-1$ matrix $P$?
\end{question}

\begin{question}
Does the sequence $\{ {m(n, P, d) \over n^{d -1 }} \}$ converge for every $d$-dimensional $0-1$ matrix $P$?
\end{question}

\begin{question}
What is the maximum possible value of ${f(n, P, d) \over n^{d -1 }}$ for all $d$-dimensional $k \times \cdots \times k$ permutation matrices $P$?
\end{question}

\section{Acknowledgments}

This research was supported by the NSF graduate fellowship under grant number 1122374.

\bibliographystyle{amsplain}

\end{document}